\documentclass[11pt,letterpaper]{amsart}
\usepackage{amssymb}
\usepackage{amsmath}
\usepackage{amsthm}
\usepackage{bbm}
\usepackage{graphicx}
\usepackage{verbatim}
\usepackage{fullpage}
\usepackage{yfonts}
\usepackage[T1]{fontenc}
\usepackage[utf8]{luainputenc}
\usepackage{appendix}
\usepackage{hyperref}

\def\FF{\mathbb{F}}

\newtheorem{theorem}{Theorem}
\newtheorem{lemma}{Lemma}

\newtheorem{corollary}{Corollary}
\newtheorem{example}{Example}
\newtheorem{remark}{Remark}

\renewcommand{\vec}[1]{\mathbf{#1}}
\newcommand{\trans}{\intercal}

\usepackage{algorithm}
\usepackage{algpseudocode}

\makeatletter
\newcommand*{\algrule}[1][\algorithmicindent]{\makebox[#1][l]{\hspace*{.5em}\vrule height .75\baselineskip depth .25\baselineskip}}%

\newcount\ALG@printindent@tempcnta
\def\ALG@printindent{%
    \ifnum \theALG@nested>0% is there anything to print
    \ifx\ALG@text\ALG@x@notext% is this an end group without any text?
    \addvspace{-3pt}% FUDGE for cases where no text is shown, to make the rules line up
    \else
    \unskip
    \ALG@printindent@tempcnta=1
    \loop
    \algrule[\csname ALG@ind@\the\ALG@printindent@tempcnta\endcsname]%
    \advance \ALG@printindent@tempcnta 1
    \ifnum \ALG@printindent@tempcnta<\numexpr\theALG@nested+1\relax% can't do <=, so add one to RHS and use < instead
    \repeat
    \fi
    \fi
}%
\usepackage{etoolbox}
\patchcmd{\ALG@doentity}{\noindent\hskip\ALG@tlm}{\ALG@printindent}{}{\errmessage{failed to patch}}
\makeatother
\makeatother
\title{On completion of a linearly independent set to a basis with shifts of a fixed vector}
\author{Marek Rychlik\\
  University of Arizona\\
  Department of Mathematics, 617 N Santa Rita Rd, P.O. Box 210089\\
  Tucson, AZ 85721-0089, USA\\}

\date{\today}

\date{\today}


\begin{document}
\maketitle
\begin{abstract}
  Let $\FF$ be an infinite field. Let $n$ be a positive integer and
  let $1\leq d\leq n$.  Let
  $\vec{f}_1,\vec{f}_2,\ldots,\vec{f}_{d-1}\in\FF^{n}$ be $d-1$
  linearly independent vectors. Let
  $\vec{x}=(x_1,x_2,\ldots,x_{d},0,0,\ldots,0)\in\FF^{n}$, with $n-d$
  zeros at the end.  Let $\vec{R}:\FF^{n}\to\FF^{n}$ be the cyclic
  shift operator to the right, e.g.
  $\vec{R}\,\vec{x}=(0,x_1,x_2,\ldots,x_{d},0,0,\ldots,0)$.  Is there
  a vector $\vec{x}\in\FF^{n}$, such that the $n-d+1$ vectors
  $\vec{x},\vec{R}\vec{x},\ldots,\vec{R}^{n-d}\vec{x}$ complete the
  set $\{\vec{f}_j\}_{j=1}^{d-1}$ to a basis of $\FF^{n}$? The answer
  is in the affirmative for every linearly independent set of
  $\vec{f}_j$, $j=1,2,\ldots,d-1$. In order to prove this fact, we
  prove that the $(n-d+1)\times(n-d+1)$ minors of the
  $(n-d+1)\times(n-d+1)$ circulant matrix
  $\begin{bmatrix}\vec{x},\vec{R}\vec{x},\ldots,\vec{R}^{n-d}\vec{x}\end{bmatrix}^\trans$
  form a Gr\"obner basis with respect to the graded reverse
  lexicographic order (grevlex).
\end{abstract}
%----------------------------------------------------------------
\section{Background}
Let $\FF$ be an infinite field. Therefore, every polynomial
$f\in\FF[x_1,x_2,\ldots,x_d]$ is either $0$ or is not identically $0$
as a function.

Let $n$ be a positive integer and let
$1\leq d\leq n$.  Let
\[ \mathcal{F}=\{\vec{f}_1,\vec{f}_2,\ldots,\vec{f}_{d-1}\}\subseteq\FF^{n} \]
be a subset of $d-1$ linearly independent vectors. In many problems we
need to complete this set with $n-d+1$ linearly independent vectors to
a basis of $\FF^{n}$. Furthermore, we may need the completing
vectors to satisfy some constraints.  For instance, we can always
pick the completing vectors from the standard basis of $\FF^{n}$.

In the current paper we constrain the completing set to shifts of a
fixed vector with a contiguous block of $n-d$ zeros.  More precisely, let
\begin{equation}
  \label{eqn:x-form}
  \vec{x}=(\underbrace{x_1,x_2,\ldots,x_{d}}_{\text{$d$ entries}},
  \underbrace{0,0,\ldots,0}_{\text{$n-d$ zeros}})\in\FF^{n}.
\end{equation}
Let $\vec{R}:\FF^{n}\to\FF^{n}$ be the cyclic shift operator to
the right. Thus for $k=0,1,\ldots,n-d$:
\[\vec{R}^k\,\vec{x}=(\underbrace{0,0,\ldots,0}_{\text{$k$ zeros}},
  \underbrace{x_1,\ldots,x_{d}}_{\text{$d$ entries}},
  \underbrace{0,0,\ldots,0}_{\text{$n-d-k$ zeros}}).
\]
The main objective of the current paper is in the theorem below:
\begin{theorem}
  \label{theorem:main}
  For every integer $n$ and $d$, $1\leq d \leq n$ and a linearly independent set 
  \[ \mathcal{F}=\left\{\vec{f}_j\right\}_{j=1}^{d-1} \subseteq\FF^n\]
  there is a vector $\vec{x}\in\FF^{n}$ of the form~\eqref{eqn:x-form}
  such that the set of $n-d+1$ vectors
  \[ \left\{ \vec{x},\vec{R}\vec{x},\ldots,\vec{R}^{n-d}\vec{x}\right\}\]
  completes the set $\mathcal{F}$ to a basis of $\FF^{n}$.
\end{theorem}
\begin{proof}
  Split into a number of lemmas that follow.
\end{proof}

The pursuit of the proof led us to considerations in ideal theory in
the ring of multivariate polynomials. In particular, we find a
Gr\"obner basis for an ideal generated by the minors of a circulant
matrix. We proceed to briefly outline the transition 
to ideal theory which allows us to prove Theorem~\ref{theorem:main}.
We will only need a minimum background from Gr\"obner basis theory
which can be found in any general reference on the subject
(e.g. \cite{cox-little-oshea, becker-weispfenning}).

The main result can be formulated in matrix form by considering the square
$n\times n$ matrix of this special form:
\[ \vec{M} = \begin{bmatrix}
    x_1 & x_2 & \ddots & x_{d} & 0 & 0 & \ddots & 0 \\
    0   & x_1 & x_2 & \ddots & x_{d} & 0 & \ddots & 0 \\    
    \ddots & \ddots & \ddots & \ddots & \ddots & \ddots &\ddots&\ddots\\
    0   & 0  & 0    & \ddots & \ddots & \ddots & x_{d} & 0\\
    0   & 0  & 0    & \ddots & \ddots & \ddots & x_{d-1} & x_{d}\\
    f_{1,1} & f_{1,2} & \ldots & \ldots & \ldots & \ldots & f_{1,n-1}& f_{1,n} \\
    f_{2,1} & f_{2,2} & \ldots & \ldots & \ldots & \ldots & f_{2,n-1}& f_{2,n} \\
    \ldots & \ldots & \ldots & \ldots & \ldots & \ldots &\ldots&\ldots\\
    f_{d-1,1} & f_{d-1,2} & \ldots & \ldots & \ldots & \ldots & f_{d-1,n-1}& f_{d-1,n} \\
  \end{bmatrix} = \begin{bmatrix}
    \vec{X}\\
    \vec{F}
  \end{bmatrix}
\]
where $\vec{X}$ is a $(n-d+1)\times n$ matrix and $\vec{F}$ is a $(d-1)\times n$ matrix.
If for every $\vec{F}$ of rank $d-1$ we can find $\vec{x}$ such that
$\vec{M}$ is non-singular then the proof of Theorem~\ref{theorem:main} follows.

We make an observation that $\vec{X}$ is a special Toeplitz, circulant matrix \cite{toeplitz-matrix}.
\begin{example}
  We  explicitly write down the matrix $\vec{M}$ for $n=6$ and $d=4$:
  \[ \vec{M} = \begin{bmatrix}
      x_1 & x_2 & x_3 & x_{4} & 0 & 0 \\
      0   & x_1 & x_2 & x_3 & x_{4} & 0\\
      0   &   0 & x_1 & x_2 & x_3 & x_{4}\\    
      f_{1,1} & f_{1,2} & f_{1,3} & f_{1,4} & f_{1,5} & f_{1,6} \\
      f_{2,1} & f_{2,2} & f_{2,3} & f_{2,4} & f_{2,5} & f_{2,6} \\
      f_{3,1} & f_{3,2} & f_{3,3} & f_{3,4} & f_{3,5} & f_{3,6}
    \end{bmatrix}.
  \]
\end{example}

We consider the matrix
\[ \vec{X} =
  \begin{bmatrix}\vec{x},\vec{R}\vec{x},\ldots,\vec{R}^{n-d}\vec{x}\end{bmatrix}^\trans
\]
obtained by putting $\vec{x}$ and its shifts in the rows of the matrix
$\vec{X}$, i.e. the $(n-d+1)\times n$ circulant matrix
\[
\vec{X} = \begin{bmatrix}
    x_1 & x_2 & \ddots & x_{d} & 0 & 0 & \ddots & 0 \\
    0   & x_1 & x_2 & \ddots & x_{d} & 0 & \ddots & 0 \\    
    \ddots & \ddots & \ddots & \ddots & \ddots & \ddots &\ddots&\ddots\\
    0   & 0  & 0    & \ddots & \ddots & \ddots & x_{d} & 0\\
    0   & 0  & 0    & \ddots & \ddots & \ddots & x_{d-1} & x_{d}
  \end{bmatrix}
\]
The symbols $x_1,x_2,\ldots,x_d$ are variables with range $\FF$.
Thus, the minors of the matrix $\vec{X}$ are polynomials in the
ring of polynomials $\FF[x_1,x_2,\ldots,x_d]$.

\begin{lemma}
  Let $\mathcal{P}_q^{n}$ denote the family of all subsets $H\subseteq\{1,2,\ldots,n\}$
  with $q$ elements.
  The determinant of the special matrix $\vec{M}$ is a homogeneous
  polynomial in variables $x_1,x_2,\ldots, x_d$ of degree $n-d$
  and
  \begin{equation}
    \label{eqn:laplace-expansion}
    \det \vec{M} = \sum_{H\in\mathcal{P}^{n}_{n-d+1}} \varepsilon^H\,\det \vec{X}_H \, \det\vec{F}_{H'}.
  \end{equation}
  where $\varepsilon^H=\pm 1$ and
  \begin{enumerate}
  \item $\vec{X}_H$ is the submatrix of $\vec{X}$ obtained by taking a subset of $n-d+1$
    columns of $\vec{X}$ with indices in the subset $H$;
  \item $\vec{F}_{H'}$ is the submatrix of $\vec{F}$ obtained from
    $\vec{F}$ by taking $d-1$ columns with indices in the complement $H'$ of
    $H$.
  \end{enumerate}
\end{lemma}

\begin{proof}
  We use Laplace's expansion of the determinant by complementary minors
  \cite{laplace-expansion}. We use the expansion along the
  first $n-d+1$ rows. 
\end{proof}

\begin{lemma}
  Let $\vec{f}_j\in\FF^n$, $j=1,2,\ldots,d-1$  be an arbitrary linearly independent set.
  Let the set
  \begin{equation}
    \label{eqn:toeplitz-minors}
    \left\{\det\vec{X}_H\right\}_{H\in \mathcal{P}_{n-d+1}^{n}}
  \end{equation}
  be a linearly independent set in $\FF_{n-d+1}[x_1,x_2,\ldots,x_d]$, the
  vector space of homogenous polynomials of degree $n-d+1$.
  Then there exists a vector $\vec{x}\in\FF^n$ with $n-d$ zeros at the end such that
  $\det\vec{M}\neq 0$. 
\end{lemma}

\begin{proof}
  By way of contradition, we assume that $\det\vec{M}=0$ for all
  $\vec{x}\in\FF^n$ whose last $n-d$ coordinates are $0$.  This
  implies that $\det\vec{M}\in \FF[x_1,x_2,\ldots,x_d]$ is a
  polynomial identically equal to $0$ as a function. As $\FF$ is an
  infinite field, this polynomial is a zero polynomial. Therefore,
  linear independence of the set~\eqref{eqn:toeplitz-minors} and
  equation~\eqref{eqn:laplace-expansion} imply that
  $\det\vec{F}_{H'}=0$ for all $H\in\mathcal{P}_{n-d+1}^n$.  Since
  $H'$ traverses all sets of cardinality $d-1$, this in turn implies
  that the rank of the matrix $\vec{F}$ is strictly less than $d - 1$,
  i.e. vectors $\vec{f}_j$, $j=1,2,\ldots,d-1$, are linearly
  dependent. This contradicts the assumptions.
\end{proof}
\begin{corollary}
  Theorem~\ref{theorem:main} will follow once we prove
  the independence of the set~\eqref{eqn:toeplitz-minors}.
\end{corollary}

\section{Ideal theory considerations}
The objective of this section is to prove
\begin{theorem}
  \label{theorem:grobner-theorem}
  The $(n-d+1)\times(n-d+1)$ minors of the circulant matrix $\vec{X}$
  form a Gr\"obner basis with respect to the graded reverse lexicographic order (grevlex).
\end{theorem}

Let $\vec{X}$ be the $(n-d+1)\times (n+1)$ matrix
\[ \vec{X} = \begin{bmatrix}
    x_1 & x_2 & \ddots & x_{d} & 0 & 0 & \ddots & 0 \\
    0   & x_1 & x_2 & \ddots & x_{d} & 0 & \ddots & 0 \\    
    \ddots & \ddots & \ddots & \ddots & \ddots & \ddots &\ddots&\ddots\\
    0   & 0  & 0    & \ddots & \ddots & \ddots & x_{d} & 0\\
    0   & 0  & 0    & \ddots & \ddots & \ddots & x_{d-1} & x_{d}
  \end{bmatrix}
\]

Let us recall that $\mathcal{P}_{n-d+1}^n$ is the family of all
subsets $H\subset \{1,2,\ldots,n\}$ of cardinality $n-d+1$.  Also,
$\vec{X}_H$ is the submatrix of $\vec{X}$ with columns whose indices
are in $H$. The set of minors
\[ \mathcal{X} = \left\{\det\vec{X}_H\,:\, H\in\mathcal{P}_{n-d+1}^n \right\} \]
is a family of homogeneous polynomials of degree $n-d+1$
in variables $x_1,x_2,\ldots,x_d$.

We will use some basic techniques from computational algebraic
geometry.  One is that of a \textbf{monomial order} or \textbf{term
  order}. The only monomial order relevant to this paper is the
\textbf{graded reverse lexicographic order} also abbreviated as
\textbf{grevlex}.  We assume that every polynomial has its terms ordered
in descending (grevlex) order. The largest term (monomial) is
called the \textbf{leading term (monomial)} of a polynomial, and will be
denoted $LM(f)$ for a polynomial $f$.  The family $LM(\mathcal{X})$ is
the set of all leading monomials $LM(f)$ of all elements
$f \in \mathcal{X}$.

\begin{lemma}
  \label{lemma:leading-monomials}
  The set $LM(\mathcal{X})$ is the set of \textbf{all} monomials of degree
  $n-d+1$ in variables $x_1,x_1,x_2,\ldots,x_d$. In particular, both
  $LM(\mathcal{X})$ and $\mathcal{X}$ itself are both bases of the
  linear vector space $\FF_{n-d+1}[x_1,x_2,\ldots,x_d]$ of homogeneous
  polynomials of degree $n-d+1$.
\end{lemma}
\begin{proof}
  Let $H=\{c_1,c_2,\ldots,c_{n-d+1}\}$ be the set of distinct $n-d+1$
  column indices ($1\leq c_j\leq n$), where the sequence $c_j$ is non
  necessarily increasing.  Using Leibniz formula for determinants
  \cite{leibniz-formula-for-determinants}, we know that the minor is
  a sum of products, up to the sign:
  \[ x_1^{k_1}\,x_2^{k_2}\cdots x_d^{k_d} \]
  where for $j=1,2,\ldots,d$:
  \[ k_j = \#\{ i\,:\, 1 \leq i \leq n-d+1,\; c_{i}-i = j-1 \}. \]
  We require that for $i=1,2,\ldots,n-d+1$.
  \begin{equation}
    \label{eqn:c-condition}
    0 \le c_{i}-i \le d-1.
  \end{equation}
  Essentially, $k_j$ represent the frequency table of the sequence
  $c_i-i+1$, $i=1,2\ldots,n-d+1$. Besides $k_j\ge 0$, we also must have
  \begin{equation}
    \label{eqn:k-condition}
    \sum_{j=1}^d k_j = n-d+1.
  \end{equation}

  Moreover, $j=c_i-i+1$ represents
  the horizontal distance from the diagonal, of the entry at the
  intersection of row $i$ with column $c_i$.  In addition, if we
  reorder the sequence $c_i$ to a sequence $c_i'=c_{\rho(i)}$, where
  $\rho:\{1,2,\ldots,n-d+1\}\to\{1,2,\ldots,n-d+1\}$ is a permutation,
  while preserving condition~\eqref{eqn:c-condition}, then the new
  sequence $k_j'$ represents another monomial in the same minor
  $\det\vec{X}_H$.  The question arises: which $\rho$ yields the
  monomial maximal in the grevlex order?  There are two objectives to
  achieve by reordering the sequence $c_j$:
  \begin{enumerate}
  \item Let $k_j>0$ be the \textbf{last} $j$ with this property,
    i.e. $k_{j+1},k_{j+2},\ldots,k_d$ are all zero; we make $j$ the
    \textbf{lowest possible}. We recall that grevlex considers
    variables to be ordered from the highest, i.e.
    $x_d>x_{d-1}>\ldots>x_1$; therefore, a monomial which has variable
    $x_j$ but no variables $x_{j+1}, x_{j+2},\ldots,x_{d}$, is
    \textbf{greater} than all polynomials which have one of those
    variables.
  \item Once we achieved the first goal, we \textbf{minimize} $k_j$
    i.e. the power of $x_j$; this is also the count of $i$ such that
    $c_i-i = j-1$. The grevlex order considers the term with the
    \textbf{lower} power $k_j$ greater.
  \end{enumerate}
  We claim that the optimal ordering is achieved when the sequence
  $c_i$ is strictly increasing. The remainder of the current proof
  is devoted to the proof of this claim.

  We note that $c_i$ is strictly increasing iff the sequence
  $j_i=c_i-i+1$ is non-decreasing.  Indeed
  \[ j_{i+1} - j_i = (c_{i+1}-(i+1)+1)-(c_i-i+1) = c_{i+1}-c_{i}-1 \ge 0.\]
  The numbers $k_j$ are simply the counts of the number of times $i$ such that $j_i=j$,
  i.e. level counts of the non-decreasing sequence $j_i$.
  Obviously, knowing $k_j$ for $j=1,2,\ldots,d$ such that
  $\sum_{j=1}^d k_j=n-d+1$ allows us to reconstruct the non-decreasing
  sequence $j_i$, $i=1,2,\ldots,n-d+1$ uniquely.
  Then the equation $j_i=c_i-i+1$ allows us to reconstruct $c_i$. Moreover,
  $c_{i+1}-c_{i}=j_{i+1}-j_i+1$, so $c_i$ is strictly increasing if $j_i$ is non-decreasing.
  Thus, we established a 1:1 correspondence between sequences $(k_j)_{j=1}^d$,
  $(j_i)_{i=1}^{n-d+1}$ and $(c_i)_{i=1}^{n-d+1}$ with the specified properties.

  It remains to be shown that the monomial obtained from an unsorted
  sequence $c_i$ is strictly smaller in grevlex order than that obtained
  from the sorted sequence. Thus, let us assume that
  the sequence
  $j_i = c_i-i+1$ is not a non-decreasing sequence; hence, for some $i$, $1\le i <n-d+1$
  $j_{i+1} < j_{i}$ or $j_{i+1}\le j_i-1$.
  Equivalently
  \[ c_{i+1}-(i+1)+1 \le (c_i-i+1)-1. \]
  Hence $c_{i+1}\le c_i$. Thus $c_{i+1}<c_{i}$ and $j_{i+1} \le j_i-2$.
  Let us consider the sequence $c_\ell'$, $\ell=1,2,\ldots,n-d+1$,
  obtained by swapping $c_i$ and $c_{i+1}$. We will prove that by
  doing so we increased the monomial $x_1^{k_1}\,x_2^{k_2}\cdots x_d^{k_d}$.
  Thus
  \[ c_{\ell}' =\begin{cases}
      c_\ell & \text{if $\ell\notin\{i,i+1\}$,}\\
      c_{i+1} & \text{if $\ell=i$,}\\
      c_{i} & \text{if $\ell=i+1$.}
    \end{cases}
  \]
  Therefore $j_\ell' = c_\ell'-\ell+1$ is given by:
  \[ j_{\ell}' = \begin{cases}
      j_\ell & \text{if $\ell\notin\{i,i+1\}$,}\\
      j_{i+1}+1 & \text{if $\ell=i$,}\\
      j_i-1\ & \text{if $\ell=i+1$.}
    \end{cases}
  \]
  Indeed,
  \[j_i' = c_i'-i+1 = c_{i+1}-(i+1)+2 = j_{i+1}+1\]
  and
  \[ j_{i+1}' = c_{i+1}'-(i+1)+1 = (c_i-i+1)-1 = j_i-1.\]
  As  $j_{i+1}+2 \le j_i$ and $j_i-1 \ge j_{i+1}+1$, all values remain in
  the range from $1$ to $d$, and the sequence $c_\ell'$ is valid
  (satisfies $1\leq c_\ell'-\ell+1 \leq d$).
  The value of $k_j$ does not change
  unless $j\in J$, where
  \[ J=\{j_{i+1},j_{i+1}+1,j_{i}-1,j_{i}\}.\]
  As $j_{i+1}+1\leq j_{i}-1$, we have
  \[j_{i+1}< j_{i+1}+1 \leq j_{i}-1 < j_i\]
  the set $J$ consists of either $3$ or $4$ values. It is
  easy to see how values $k_j$ change for $j\in J$. There are 2 cases.

  \noindent\textbf{Case $j_i -1 = j_{i+1}+1$:} when $J$ consists of $3$ elements.
  Then
  \begin{align*}
    k_{j_{i+1}}' &= k_{j_{i+1}} - 1,\\ 
    k_{j_i-1}'&=k_{j_i-1}+2,\\
    ( k_{j_{i+1}+1}'&=k_{j_{i+1}+1}+2 )\\
    k_{j_i}'&=k_{j_i}-1.
  \end{align*}
  The parenthesized third equation is a copy of the second equation.

  \noindent\textbf{Case $j_{i+1}+1 < j_i-1$:} when $J$ consists of $4$ elements.
  Then
  \begin{align*}
    k_{j_{i+1}}' &= k_{j_{i+1}} - 1,\\ 
    k_{j_i-1}'&=k_{j_i-1}+1, \\
    k_{j_{i+1}+1}'&=k_{j_{i+1}+1}+1,\\
    k_{j_i}'&=k_{j_i}-1.
  \end{align*}
  Since the power of the highest-indexed impacted
  variable $x_{j_i}$ drops, the new sequence $k_\ell'$,
  $\ell=1,2,\ldots,n-d+1$ corresponds to a monomial which is
  \textbf{strictly greater} in grevlex order. This
  proves that the maximum monomial is obtained when $c_\ell$ is sorted 
  in ascending order.

  This completes the proof of the main claim, and the entire lemma.
\end{proof}

\begin{remark}
  It follows that $LM\left(\det\vec{X}_H\right)$ with respect to the grevlex
  order is the product of the diagonal entries of $\vec{X}_H$.
\end{remark}

\begin{remark}
  The set $\mathcal{X}$ is a Gr\"obner basis of the ideal generated by
  $\mathcal{X}$.  This follows from
  Lemma~\ref{lemma:leading-monomials} and the Buchberger
  criterion. This ideal is also $\mathfrak{m}^{n-d+1}$ where
  $\mathfrak{m} = \langle x_1,x_2,\ldots,x_d\rangle$ is the maximal
  ideal generated by $x_1,x_2,\ldots,x_d$ and is the ideal of
  homogeneous polynomials vanishing at $0$.
\end{remark}

\begin{example}
  We give a counterexample to the conjecture that the grevlex order
  could be substituted with the (graded) lexicographic order
  (lex or grlex). Let the matrix be:
  \[ \vec{X} = \begin{bmatrix}
      x_1 & x_2 & x_3 & 0 \\
      0   & x_1 & x_2 & x_3
    \end{bmatrix}.
  \]
  Thus $d=3$ and $n=4$. We find:
  \[ \mathcal{X} = \left\{x_1^2, x_1\,x_2, x_1\,x_3,x_2^2-x_1\,x_3,x_2\,x_3,x_3^2\right\}.\]
  We also find that with grlex or lex used instead of grevlex:
  \[ LM_{lex}(\mathcal{X}) = \left\{x_1^2,x_1\,x_2,x_1\,x_3,x_2\,x_3,x_3^2\right\}. \]
  Notably, $x_2^2$ is missing. We note that $x_2^2<x_1\,x_3$ with respect to the (graded)
  lexicographic order, but $x_2^2 > x_1\,x_3$ with respect to grevlex. Therefore,
  \begin{align*}
    LM_{grlex}\left(x_2^2-x_1\,x_3\right)&=x_1\,x_3,\\
    LM_{grevlex}\left(x_2^2-x_1\,x_3\right)&=x_2^2.\\
  \end{align*}
  Hence, consistent with the statement of Lemma~\ref{lemma:leading-monomials},
  \[ LM_{grevlex}(\mathcal{X}) = \left\{x_1^2,x_1\,x_2,x_1\,x_3,x_2^2,x_2\,x_3,x_3^2\right\}. \]
\end{example}
\section{Generalizations}
The assumption that $\FF$ is an infinite field can be relaxed.
If we assume that $\FF$ is an infinite integral domain then
the only identically vanishing polynomial is the zero polynomial.
This is sufficient to prove Theorem~\ref{theorem:main}.

Theorem~\ref{theorem:grobner-theorem} and
Lemma~\ref{lemma:leading-monomials} are valid over any ring with unity.

\section{Acknowledgments}
I thank Johnatan Ashbrock for the question that led to
Theorem~\ref{theorem:main}.

% ----------------------------------------------------------------
\bibliographystyle{plain}
\bibliography{db.bib}

\begin{thebibliography}{1}

\bibitem{becker-weispfenning}
Thomas Becker and Volker Weispfenning.
\newblock {\em Gr{\"o}bner Bases}.
\newblock Springer, 1993.

\bibitem{cox-little-oshea}
David~A. Cox, John Little, and Donal O{\rq}Shea.
\newblock Springer, 2012.

\bibitem{laplace-expansion}
Wikipedia.
\newblock Laplace expansion of a determinant by complementary minors.
\newblock
  \href{https://en.wikipedia.org/wiki/Laplace\_expansion\#Laplace\_expansion\_of\_a\_determinant\_by\_complementary\_minors}{
  Laplace expansion of a determinant by complementary minors}, May 2019.

\bibitem{leibniz-formula-for-determinants}
Wikipedia.
\newblock Leibniz formula for determinants.
\newblock
  \href{https://en.wikipedia.org/wiki/Leibniz\_formula\_for\_determinants}{Leibniz
  formula for determinants}, May 2019.

\bibitem{toeplitz-matrix}
Wikipedia.
\newblock Toeplitz matrix.
\newblock \href{https://en.wikipedia.org/wiki/Toeplitz\_matrix}{Toeplitz
  Matrix}, May 2019.

\end{thebibliography}

%\begin{thebibliography}{99}
%\end{thebibliography}
%----------------------------------------------------------------
%\newpage
%----------------------------------------------------------------
%\appendix
%----------------------------------------------------------------
\end{document}